\documentclass[10pt]{amsart}
\usepackage{amssymb}

\newcommand{\res}{\upharpoonright}
\newcommand{\N}{{\mathbb N}}

\theoremstyle{plain}
\newtheorem{theorem}{Theorem}

\newtheorem{lemma}[theorem]{Lemma}
\newtheorem{proposition}[theorem]{Proposition}

\theoremstyle{remark}

\newtheorem{claim}{Claim}

\begin{document}

\title[Ramsey theorem for partial orders]{A Ramsey theorem for partial orders\\ with linear extensions}

\author{S{\l}awomir Solecki \and Min Zhao}

\address{Department of Mathematics\\
University of Illinois\\
1409 W. Green St.\\
Urbana, IL 61801, USA}

\email{ssolecki@math.uiuc.edu}

\email{minzhao1@math.uiuc.edu}

\thanks{Research supported by NSF grant DMS-1266189.}


\begin{abstract} We prove a Ramsey theorem for finite sets equipped with a partial order and a fixed number of linear orders extending the partial order.
This is a common generalization of two recent Ramsey theorems due to Soki{\'c}. As a bonus, our proof gives new arguments for these two results.
\end{abstract}

\maketitle

\section{The theorem}\label{S:theo}

In recent years, there has been a renewed interest in Structural Ramsey Theory sparked by the discovery in \cite{KPT} of connections between this area 
and Topological Dynamics. Paper \cite{NVT} gives a survey of these developments. In this context, some attention was directed towards the so-called mixed structures  
obtained by superimposing a number of simpler structures that are known to be Ramsey; see \cite[Section 5.7]{NVT}. A general Ramsey theorem for such structures 
was proved in \cite{Bo} (see also \cite{Sok3}) under the additional assumption that the superimposed structures are independent from each other. The present work contributes 
a particular structural Ramsey theorem to this area, where the superimposed structures are not independent, but rather are interconnected in a natural way.

In this paper, all orders are strict orders.

{\em For the rest of the paper, we fix a natural number $p>0$.}

By a {\em structure} we understand a set $X$ equipped with a partial order $P$ and $p$ linear orders $L_0, \dots, L_{p-1}$ each of which extends $P$. We write
\[
\vec{L}
\]
for $(L_0, \dots, L_{p-1})$ and
\[
(X, P, \vec{L})
\]
for the whole structure. A structure is called finite if $X$ is a finite set. Given two structures ${\mathcal X} = (X, P^X, \vec{L}^X)$ and ${\mathcal Y}= (Y, P^Y, \vec{L}^Y)$,
a function $f\colon X\to Y$ is an {\em embedding} if for all $x_1, x_2\in X$
\[
x_1P^Xx_2\Longleftrightarrow f(x_1)P^Y f(x_2)
\]
and, for each $i<p$,
\[
x_1L_i^Xx_2\Longleftrightarrow f(x_1)L_i^Y f(x_2).
\]
By a {\em copy} we understand the image of an embedding.

For a natural number $d>0$, a {\em $d$-coloring} is a coloring with $d$ colors.

\begin{theorem}\label{T}
Let $d>0$, and let ${\mathcal X}= (X, P^X, \vec{L}^X)$ and ${\mathcal Y}= (Y, P^Y, \vec{L}^Y)$ be finite structures.
There exists a finite structure ${\mathcal Z} = (Z, P^Z, \vec{L}^Z)$ with the following property: for each $d$-coloring of all copies
of $\mathcal X$ in $\mathcal Z$, there exists a copy ${\mathcal Y}'$ of $\mathcal Y$ in $\mathcal Z$ such that all copies of $\mathcal X$
in ${\mathcal Y}'$ have the same color.
\end{theorem}

The theorem above gives a common generalization of the following two of its known special cases.

The first one is the case $p=1$, that is, the case when structures are equipped with a partial order and a single linear order extending it.
This case was proved by Soki{\'c} \cite[Theorem 7(6)]{Sok} using results of Paoli, Trotter and Walker \cite{PTW} and Fouch{\'e} \cite{Fo}.
Because of certain peculiar features of Soki{\'c}'s argument (for example, the usage of the ordering property to prove the Ramsey property),
there has been some interest in finding a more direct proof. Our argument for Theorem~\ref{T} specialized to the case $p=1$ gives just such a short and
direct proof.

The second case is the case of finite sets endowed only with $p$ linear orders. This situation corresponds to
$P^X=P^Y=\emptyset$ (when one can obviously make $P^Z=\emptyset$) in Theorem~\ref{T}. It was proved by Soki{\'c} in \cite[Theorem 10]{Sok2}.
Our proof here also specializes to an argument different from the one in \cite{Sok2}.

In our proofs, we use some ideas from \cite{Fo} and \cite{PTW}. We connect them with a special case of the main theorem from \cite{So}.

The proof of Theorem~\ref{T} is structured as follows. In Section~\ref{S:prodth}, we prove a product Ramsey theorem that is the Ramsey theoretic core of 
Theorem~\ref{T}. In Sections~\ref{S:lira} and \ref{S:str},
we make explicit certain canonical structures and morphisms important to the proof. Once these structures are properly defined and their natural properties are
established, the theorem is proved by appropriately interpreting the objects involved in it and applying the product Ramsey theorem from Section~\ref{S:prodth}.
This is done in Section~\ref{S:pr}. Section~\ref{S:redo} has an explanatory character. In it, we make precise the relationship between the product Ramsey theorem and Theorem~\ref{T}
using general notions introduced \cite{Sol}.

\section{A product Ramsey theorem}\label{S:prodth}

As promised in Section~\ref{S:theo}, we prove here a product Ramsey result, Proposition~\ref{P:VProduct}, needed in our proof of Theorem~\ref{T}.  
We establish it as a consequence of two known Ramsey theorems.

We adopt the notational convention that each natural number is equal to the set
of its predecessors, that is,
\[
m = \{ i\colon i<m\}.
\]
In particular, $0 = \emptyset$. The set $m$ is considered to be linearly ordered with its natural order inherited from $\N$.
For a set $X$ and a natural number $k$,
\[
\binom{X}{k}
\]
is the family of all $k$ element subsets of $X$. The set $X$ can itself be a natural number $m$ and then $\binom{m}{k}$ is the family of all $k$
element subsets of $m$.

We formulate all our results in terms of rigid surjections, rather than partitions, as this form fits the applications better; see Lemma~\ref{L:rig} and the proof of Lemma~\ref{L:two}(ii).
Let $A$, $B$ be two finite linearly ordered sets.
A function $r\colon B\to A$ is a {\em rigid surjection} if it is a surjection and the images of initial segments of $B$ are initial segments of $A$, in other words,
if for all $a_1, a_2\in A$, with $a_1$ preceding $a_2$ in $A$, we have that $a_1$ is
first attained by $r$ before $a_2$ is first attained by $r$. See \cite{Sol} for information on the language of rigid surjections.

Recall that we have fixed a natural number $p>0$. A sequence $\vec{a} = (a_0, \dots, a_{p-1})$ of length $p$ of elements of $A$ is called {\em anchored} if
$a_0$ is the smallest element of $A$.

We will be considering linearly ordered sets $A$ and $B$ with anchored sequences $\vec{a} = (a_0, \dots, a_{p-1})$ in $A$ and
$\vec{b} = (b_0, \dots, b_{p-1})$ in $B$.
Let
\[
\binom{B, \vec{b}}{A, \vec{a}}_{\rm rs}
\]
be the set of all rigid surjections $r\colon B \to A$ such that $r(b_i)=a_i$ for all $i <p$.
Note that having anchored sequences $\vec{a} = (a_0, a_1, \dots, a_{p-1})$ and
$\vec{b} = (b_0, b_1\dots, b_{p-1})$ is equivalent, in this context, to having arbitrary sequences
$(a_1, \dots, a_{p-1})$ and $(b_1\dots, b_{p-1})$ since $r$ automatically maps the smallest element of $A$ to the smallest element of $B$. However, in view of our applications
in Section~\ref{S:str}, it will be notationally convenient to keep the elements $a_0$ and $b_0$ in the sequences.

Let $m$ be a natural number. Let $\vec{i} = (i_0, \dots, i_{p-1})$ be an anchored sequence of elements of $m$.
For finite subsets $S_0, \dots, S_{m-1}, T_0, \dots, T_{m-1}$ of $\N$ and $s\in \binom{m, \vec{i}}{A, \vec{a}}_{\rm rs}$ and $t\in \binom{m, \vec{i}}{B, \vec{b}}_{\rm rs}$, we write
\[
(S_0, \dots, S_{m-1},s)\ll  (T_0, \dots, T_{m-1}, t)
\]
if for each $i< m$, $S_i\subseteq T_i$ and there is $r\in \binom{B, \vec{b}}{A, \vec{a}}_{\rm rs}$ with $s=r\circ t$.

\begin{proposition}\label{P:VProduct} Assume we are given $d>0$, finite linearly ordered sets $A, B$, anchored
sequences $\vec{a}$ and $\vec{b}$ of length $p$ of elements of $A$
and $B$, respectively, and two natural numbers $k, l$. Then there exist natural numbers $m, n$  and an anchored sequence $\vec{i}$ of length $p$ of elements of $m$ such that for
each $d$-coloring of
$\binom{n}{k}^m\times \binom{m, \vec{i}}{A, \vec{a}}_{\rm rs}$
there exists $(T_0, \dots, T_{m-1}, t) \in \binom{n}{l}^m\times \binom{m, \vec{i}}{B, \vec{b}}_{\rm rs}$ such that
\begin{equation}\notag
\{ (S_0, \dots, S_{m-1}, s) \in \binom{n}{k}^m\times \binom{m, \vec{i}}{A, \vec{a}}_{\rm rs}\colon (S_0, \dots, S_{m-1},s)\ll  (T_0, \dots, T_{m-1}, t)\}
\end{equation}
is monochromatic.
\end{proposition}

Proposition \ref{P:VProduct} is a quick consequence of two known Ramsey statements, which we now recall.
The first statement is the product of the classical Ramsey theorem,  see \cite{GRS}.
For $S_0, \dots, S_{m-1}, T_0, \dots, T_{m-1}$ finite subsets of $\N$, we write
\[
(S_0, \dots, S_{m-1})\leq (T_0, \dots, T_{m-1}),
\]
if for each $i< m$, $S_i\subseteq T_i$.

\smallskip

\noindent {Product Ramsey Theorem.}
{\em Given $d>0$ and natural numbers $k, l, m$, there exists a natural number $n$ such that for each $d$-coloring of $\binom{n}{k}^m$
there exists $(T_0, \dots, T_{m-1}) \in \binom{n}{l}^m$ such that
\[
\{ (S_0, \dots, S_{m-1}) \in \binom{n}{k}^m\colon  (S_0, \dots, S_{m-1})\leq  (T_0, \dots, T_{m-1})\}
\]
is monochromatic.}

\smallskip

The following result is a particular case of \cite[Theorem 1]{So}. (One considers \cite[Theorem 1]{So} for
the language consisting of $p-1$ constants, that is, $p-1$ function symbols of arity $0$.) The case $p=1$ of this result is just the dual Ramsey theorem.

\smallskip
\noindent {Dual Ramsey Theorem with Constants.}
{\em Assume we are given $d>0$ and finite linearly ordered sets $A, B$ with anchored sequences $\vec{a}$ and $\vec{b}$ of length $p$ in $A$ and $B$,
respectively.
Then there exist a natural number $m$ and an anchored sequence $\vec{i}$ of length $p$ of elements of $m$ such that for each $d$-coloring of
$\binom{m,\vec{i}}{A,\vec{a}}_{\rm rs}$
there exists $t \in \binom{m,\vec{i}}{B, \vec{b}}_{\rm rs}$ with
\[
\{ s\circ t\colon s\in \binom{B,\vec{b}}{A,\vec{a}}_{\rm rs}\}
\]
monochromatic.}

\begin{proof}[Proof of Proposition~\ref{P:VProduct}]

Choose $(m,\vec i)$ in terms of $ d,\,p,\,(A,\vec a),\, (B,\vec b)$ so the Dual Ramsey Theorem with Constants holds for $(m,\vec i)$.
Let $n$ be large in terms of $A,\, d,\,  k,\, l,\, m$ so the Product Ramsey Theorem with $d^{|\binom{m,\vec{i}}{A,\vec{a}}_{\rm rs}|}$
colors holds for $n$.

Let $\phi$ be a coloring with $d$ colors of $\binom{n}{k}^m\times  \binom{m,\vec{i}}{A,\vec{a}}_{\rm rs}$.
Let $\psi$ be a coloring with $d^{|\binom{m,\vec{i}}{A,\vec{a}}_{\rm rs}|}$ colors of $\binom{n}{k}^m$ such that for each
$(S_0,\dots,S_{m-1})$, $(S_0^{\prime},\dots,S_{m-1}^{\prime})\in \binom{n}{k}^m$
\[
\begin {split}
&\psi(S_0,\dots,S_{m-1})=\psi(S_0^{\prime},\dots,S_{m-1}^{\prime}) \Longleftrightarrow \\
&\forall s\in \binom{m,\vec{i}}{A,\vec{a}}_{\rm rs}\; \phi(S_0,\dots,S_{m-1},s)=\phi(S_0^{\prime},\dots,S_{m-1}^{\prime},s).
\end {split}
\]
Then by the choice of $n$, there exists $(T_0,\dots,T_{m-1})\in {\binom{n}{l}}^m$, such that $\psi$ is constant on
\[
\{(S_0,\dots,S_{m-1})\in \binom{n}{k}^m\colon (S_0,\dots,S_{m-1})\leq (T_0,\dots,T_{m-1}) \},
\]which implies for $(S_0,\dots,S_{m-1})\in \binom{n}{k}^m$ with $(S_0,\dots,S_{m-1})\leq (T_0,\dots,T_{m-1})$, the color
$\phi(S_0,\dots,S_{m-1},s)$ only depends on $s$.
Then by the choice of $(m,\vec i)$, there exists $t\in \binom{m,\vec{i}}{B, \vec{b}}_{\rm rs}$ such that
$\phi$ is constant on the set from the conclusion of the proposition.
\end {proof}

\section{Linear orders and a twisted product Ramsey theorem}\label{S:lira}

The point of this section is to obtain a reformulation of Proposition~\ref{P:VProduct} that introduces a twist to the product. 

First, we need to define new objects. 
Let $K$ be a linear order on a set $X$, as usual assumed to be a strict order, and let $x\in X$. Put
\begin{equation}\label{E:cut}
(K)_x = (\{ y\in X\colon yKx\}, K\res \{ y\in X\colon yKx\}).
\end{equation}

Let $L$ be a linear order on a finite set $Y$. By
\begin{equation}\label{E:linn}
{\rm lin}_L
\end{equation}
we denote the set of all linear orders on $Y$, which we order as follows.
Let $L_1, L_2\in {\rm lin}_L$. We put $L_1$ {\em below} $L_2$ if there exist $x,y\in Y$ such that
$(L_1)_x= (L_2)_y$ and $xLy$. (By $(L_1)_x= (L_2)_y$ here we mean the literal equality, not just an isomorphism.) In other words, let $|Y|=n$ and let $(x_i)_{i< n}$ and
$(y_i)_{ i<n}$ be enumerations of $Y$ in the $L_1$- and $L_2$-increasing order, respectively. We put $L_1$ below $L_2$ if
$(x_i)_{i<n}$ is smaller than $(y_i)_{i< n}$ in the lexicographic order with respect to $L$.

The proof of the following lemma is straightforward.

\begin{lemma}\label{L:lin}
${\rm lin}_L$ is linearly ordered by the above defined relation and $L$ is its smallest element.
\end{lemma}

Assume we are given a natural number $m$ and $B\subseteq {\rm lin}_L$.
Let $\vec i$, $\vec b$ be anchored sequences of length $p$ of elements of $m$ and $B$, respectively.
For
\begin{equation}\label{E:tT}
\tau= (T_0,\dots,T_{m-1},t)\in {{\N}\choose{|Y|}}^m\times \binom{m, \vec{i}}{B, \vec b}_{\rm rs}
\end{equation}
and $i<m$, let
\begin{equation}\label{E:pii}
\pi^\tau_i:(Y,t(i))\to (T_i,<\upharpoonright T_i)
\end{equation}
be the unique isomorphism.
Assume we are additionally given a linear order $K$ on a finite set $X$,
$A\subseteq {\rm lin}_K$, and an anchored sequence $\vec a$ of length $p$ of elements of $A$.
Let $\tau$ be as in \eqref{E:tT} and let
\[
\sigma= (S_0,\dots,S_{m-1},s)\in  {{Y}\choose{|X|}}^m\times \binom{B, \vec b}{A, \vec a}_{\rm rs}.
\]
Define
\begin {equation}\label {E:TwistComp}
\tau\cdot \sigma =(\pi^\tau_0(S_0),\dots,\pi^\tau_{m-1}(S_{m-1}),\, s\circ t) \in  {{\N}\choose{|X|}}^m\times \binom{m, \vec{i}}{A, \vec a}_{\rm rs}.
\end {equation}

If $n$ is a natural number taken with the linear order $<\res \N$ inherited from $\N$, we let
\[
{\rm lin}_n  = {\rm lin}_{<\res n}.
\]
Consider the situation when $(X, K)$ is the natural number $k$ with the natural order and $(Y,L)$ is the natural number
$l$ with the natural order.
Note that directly from \eqref{E:TwistComp}, $\tau\cdot \sigma \ll \tau$, so the following result is an immediate consequence of Proposition~\ref{P:VProduct}.

{\em
Assume we are given $d>0$, and natural numbers $k, l$. Let $A\subseteq {\rm lin}_k$ and $B\subseteq {\rm lin}_l$, and
let $\vec a$, $\vec b$ be anchored sequences of length $p$ of elements of $A$ and $B$, respectively. Then there exist natural numbers $m, n$  and an anchored sequence
$\vec{i}$ of length $p$ of elements of $m$ such that for each $d$-coloring of
${{n}\choose{k}}^m\times \binom{m, \vec{i}}{A, \vec a}_{\rm rs}$
there exists $\tau_0 \in {{n}\choose{l}}^m\times \binom{m, \vec{i}}{B, \vec b}_{\rm rs}$ such that
\begin{equation}\notag
\{ \tau_0\cdot\sigma \colon \sigma \in {{l}\choose{k}}^m\times \binom{B, \vec b}{A, \vec a}_{\rm rs}\}
\end{equation}
is monochromatic.}

Since arbitrary finite linear orders $(X,K)$ and $(Y,L)$ can be identified with $k$ and $l$, respectively, the result above can be restatement as Proposition~\ref{P:TwistProduct} below.

\begin{proposition}\label{P:TwistProduct}
Assume we are given $d>0$, and linear orders $K$, $L$ on finite sets $X$ and $Y$, respectively.
Let $A\subseteq {\rm lin}_K$ and $B\subseteq {\rm lin}_L$, and
let $\vec a$, $\vec b$ be anchored sequences of length $p$ of elements of $A$ and $B$, respectively. Then there exist natural numbers $m, n$  and an anchored sequence
$\vec{i}$ of length $p$ of elements of $m$ such that for each $d$-coloring of ${{n}\choose{|X|}}^m\times \binom{m, \vec{i}}{A, \vec a}_{\rm rs}$
there exists $\tau_0 \in {{n}\choose{|Y|}}^m\times \binom{m, \vec{i}}{B, \vec b}_{\rm rs}$ such that
\begin{equation}\notag
\{ \tau_0\cdot \sigma \colon \sigma \in {{Y}\choose{|X|}}^m\times \binom{B, \vec b}{A, \vec a}_{\rm rs}\}
\end{equation}
is monochromatic.
\end{proposition}

\section{Certain canonical structures}\label{S:str}

In this section, we define certain concrete structures and prove their basic properties. These structures are essentially the ones we need for the conclusion of 
Theorem~\ref{T}.   

For the remained of this section, $P$ is a partial order on a finite set $Y$, and $L$ is a linear order on $Y$ extending $P$.
Let
\[
{\rm lin}_L(P)\subseteq {\rm lin}_L
\]
be the set of all linear orders of $Y$ extending $P$. The set ${\rm lin}_L(P)$ is equipped with the linear order inherited from ${\rm lin}_L$.
Let $X\subseteq Y$. Note that the linear order $L\res X$ extends the partial order $P\res X$. Define
\[
{\rm res}_X\colon {\rm lin}_L(P)\to {\rm lin}_{L\res X}(P\res X),\; {\rm res}_X(L') = L'\res X.
\]

Now, in addition to $Y$, $P$, and $L$, we fix linear orders $L_1, \dots, L_{p-1}$ on $Y$ that extend $P$, and let
\[
\vec{L} = (L,L_1, \dots, L_{p-1}).
\] By Lemma~\ref{L:lin}, $\vec{L}$ is an anchored sequence in ${\rm lin}_L(P)$. We set \[
\vec{L}\res X= (L\res X, L_1\res X, \dots, L_{p-1}\res X).
\]

The following lemma is essentially \cite[Lemma 14]{PTW}. We include a proof of it for completeness.

\begin{lemma}\label{L:rig}
${\rm res}_X$ is an element of ${{\rm lin}_L(P), \vec {L} \choose {\rm lin}_{L\res X}(P\res X), \vec{L}\res X}_{\rm rs}$.
\end{lemma}
\begin {proof}
By the definition of ${\rm res}_X$, it suffices to show that ${\rm res}_X$ is a rigid surjection from ${\rm lin}_L(P)$ to ${\rm lin}_{L\res X}(P\res X)$.
Recall \eqref{E:cut}.

Fix $L_1, L_2\in {\rm lin}_L(P)$, $M_1\in {\rm lin}_{L\res X}(P\res X)$, and $x,y\in Y$ with  $(L_1)_x= (L_2)_y$.
Assume that $L_1$ is the smallest element of ${\rm lin}_L(P)$ such that $L_1\res X= M_1$.

\begin{claim}\label{Cl1}
If $y\not\in X$, then $xLy$ or $x=y$.
\end{claim}

\noindent Proof of Claim~\ref{Cl1}.  Towards a contradiction, assume that $yLx$. Define a linear order $L_1'$ on $Y$ by
\begin{enumerate}
\item[(a)] $L_1'\res (Y\setminus \{ y\}) = L_1\res (Y\setminus \{ y\})$;

\item[(b)] $y$ is the $L_1'$-immediate predecessor of $x$.
\end{enumerate}

Note that $L_1'$ extends $P$. Indeed, since $L_1$ extends $P$, condition (a) is compatible with $P$. Also we
have
\[
(L_1')_y= (L_1)_x = (L_2)_y.
\]
So for $z\not= y$, if $z\in (L_1')_y$, then $z\in (L_2)_y$, and if $z\not\in (L_1')_y$, then $z\not\in (L_2)_y$, therefore, since $L_2$ extends $P$,
condition (b) is compatible with $P$. Thus, $L_1'\in {\rm lin}_L(P)$. We have that $L_1'$ is below $L_1$ in ${\rm lin}_L(P)$ since $(L_1')_y= (L_1)_x $
and $yLx$. Since $y\not\in X$, $X\subseteq Y\setminus \{ y\}$, so by (a)
\[
L_1'\res  X = L_1\res X = M_1,
\]
contradicting the choice of $L_1$ and proving the claim.

\begin{claim}\label{Cl2} If $x\not\in X$, $xL_1y$, and there is no $z\in X$ with $xL_1zL_1y$, then $xLy$.
\end{claim}

\noindent Proof of Claim~\ref{Cl2}. Note that by assumption $x\not= y$, so if the conclusion fails, then $yLx$. There are $z_1, z_2$ such that
\begin{enumerate}
\item[(i)] $(xL_1z_1\hbox{ or }x=z_1)$ and $z_1L_1y$;

\item[(ii)] $z_1$ is an $L_1$-immediate predecessor of $z_2$;

\item[(iii)] $z_2Lz_1$.
\end{enumerate}
To get such $z_1$ and $z_2$, let $x=v_0, v_1, \dots , v_k=y$ be such that $v_i$ is the $L_1$-immediate predecessor of
$v_{i+1}$ for $i<k$. If for each $i<k$, $v_i L v_{i+1}$, then we would have $xLy$ contradicting $yLx$.
So for some $i<k$, $v_{i+1} L v_i$, and we take $z_2=v_{i+1}$ and $z_1=v_i$. 

Note that, by (i) and by our assumptions,  $z_1\not\in X$.

Define a linear order $L_1'$ on $Y$ by
\begin{enumerate}
\item[(a)] $L_1'\res (Y\setminus \{ z_1\}) = L_1\res (Y\setminus \{ z_1\})$;

\item[(b)] $z_2$ is the $L_1'$-immediate predecessor of $z_1$.
\end{enumerate}

The linear order $L_1'$ extends $P$. Indeed, since $L_1$ extends $P$, condition (a) is compatible with $P$; by (ii) and (iii),
condition (b) is compatible with $P$ as $L$ and $L_1$ extend $P$.
So $L_1'\in {\rm lin}_L(P)$. Since $(L_1')_{z_2} = (L_1)_{z_1}$ and $z_2Lz_1$, $L_1'$ is below $L_1$. Since $z_1\not\in X$,
we get $L_1'\res X = L_1\res X = M_1$ contradicting our choice of $L_1$ and proving the claim.

\smallskip

Now assume that $x\not= y$. Let $M_2= L_2\res X$ and assume that $M_1$ is below $M_2$ in ${\rm lin}_{L\res X}(P\res X)$. We
need to show that $L_1$ is below $L_2$ in ${\rm lin}_L(P)$.

If $y\not \in X$, by Claim~\ref{Cl1}, we have $xLy$, so $L_1$ is below $L_2$, as required.

So assume $y\in X$. If $x\in X$, then $(M_1)_x = (M_2)_y$ and $x\not= y$. So $xLy$ by our assumption that $M_1$ is below $M_2$.
Thus, $L_1$ is below $L_2$ as required.

So assume that $y\in X$ and $x\not\in X$. Let $y'\in X$ be such that $xL_1y'$ and $z\not\in X$ for all $xL_1zL_1y'$. Such a $y'$ exists
since $xL_1y$ (as $(L_1)_x = (L_2)_y$ and $x\not= y$) and $y\in X$. By Claim~\ref{Cl2}, $xLy'$. If $yLx$, then $yLy'$. Note that
$(M_1)_{y'} =(M_2)_y$ since $(L_1)_x = (L_2)_y$. So we have that $M_2$ is below $M_1$,
contradiction. Thus, $xLy$ and $L_1$ is below $L_2$, as required.
\end {proof}

The set $\N$ is equipped with its natural linear order, which we denote by $<$. Let $m$ be a natural number.
We define a partial order $<_{\rm pr}$ on $\N^m$ by letting
\[
(k_0, \dots, k_{m-1})<_{\rm pr} (l_0, \dots, l_{m-1})
\]
if and only if $k_i<l_i$ for each $0\leq i< m$. For $i<m$, let
$<_{{\rm lx},i}$ be the linear order in $\N^m$ defined by letting
\[
(k_0, \dots, k_{m-1})<_{{\rm lx},i} (l_0, \dots, l_{m-1})
\]
if and only if
there exists $j\geq 0$ such that $k_{i+_m j}<l_{i+_m j}$ and  $k_{i+_m j'}=l_{i+_m j'}$ for all $0\leq j'<j$, where $+_m$ stands for addition modulo $m$.
In particular, $<_{{\rm lx},0}$ is the usual lexicographic order.
Note that each $<_{{\rm lx},i}$ extends $<_{\rm pr}$.

Fix an anchored sequence
\[
\vec{i} = (i_0, \dots, i_{p-1}).
\]
of elements of $m$. Let
\[
\vec{<}_{{\rm lx}, \vec{i}} = (<_{{\rm lx}, i_0}, \dots, <_{{\rm lx}, i_{p-1}}).
\]
Then
\[
(\N^m, <_{\rm pr}, \vec{<}_{{\rm lx}, \vec{i}})
\]
is a structure.

Let $\tau \in {{\N}\choose{|Y|}}^m\times \binom{m, \vec{i}}{{\rm lin}_L(P), \vec{L}}_{\rm rs}$. Recall \eqref{E:pii} and define
\begin {equation}\label{E:pitau}
\pi^\tau\colon Y\to \N^m,\; \pi^\tau(y) = (\pi^\tau_0(y), \dots, \pi^\tau_{m-1}(y)).
\end {equation}

\begin{lemma}\label{L:two}
\begin{enumerate}
\item[(i)] $\pi^\tau$ is an embedding from $(Y, P, \vec{L})$ to $(\N^m, <_{\rm pr}, \vec{<}_{{\rm lx}, \vec{i}})$.

\item[(ii)] Let $X'\subseteq \pi^\tau(Y)$. Then, for $X= (\pi^\tau)^{-1}(X')$, we have
\[
X'= \pi^{\tau\cdot \sigma}(X),
\]
for some $\sigma \in   {{Y}\choose{|X|}}^m\times \binom{{\rm lin}_{L}(P), \vec{L}}{{\rm lin}_{L\res X}(P\res X), \vec{L}\res X}_{\rm rs}$.
\end{enumerate}
\end{lemma}

\begin{proof} (i) Since each partial order is the intersection of all the linear orders containing it, we have that, for $y_1, y_2\in Y$,
\[
y_1Py_2\Longleftrightarrow y_1t(i)y_2 \hbox{ for all } i< m.
\]
It follows that $\pi^\tau$ preserves $P$. Since
\[
t(0)= L,\, t(i_1)=L_1,\dots ,\, t(i_{p-1}) = L_{p-1}
\]
we see that $\pi^\tau$ preserves each linear order in $\vec{L}$.

(ii) Let
\[
\sigma = ((\pi^\tau_0)^{-1}(p_0(X')), \dots, (\pi^\tau_{m-1})^{-1}(p_{m-1}(X')), {\rm res}_X)
\]
where $p_i$, $i< m$, is the $i$-th projection from $\N^m$ to $\N$. By Lemma~\ref{L:rig}, we have
$\sigma \in {{Y}\choose{|X|}}^m\times \binom{{\rm lin}_{L}(P), \vec{L}}{{\rm lin}_{L\res X}(P\res X), \vec{L}\res X}_{\rm rs}$.
The remainder of the conclusion, follows from the observation, made by a direct computation, that for $i<m$
\[
\pi_i^{\tau\cdot \sigma} = \pi_i^\tau\res X. \qedhere
\]
\end{proof}

\section{Proof of Theorem~\ref{T}}\label{S:pr}

Let ${\mathcal X}= (X, P^X, \vec{L}^X)$ and ${\mathcal Y} = (Y, P^Y, \vec{L}^Y)$ be given. We assume, as we can, that $\mathcal X$ is a substructure of
$\mathcal Y$. Fix the number of colors $d$. Set $K=L_0^X$, $L= L_0^Y$,
$A={\rm lin}_{L_0^X}(P^X)$, $B= {\rm lin}_{L_0^Y}(P^Y)$, $\vec{a}= \vec{L}^X$ and
$\vec{b}= \vec{L}^Y$. Apply Proposition~\ref{P:TwistProduct} to this data obtaining
$m, n$ and $\vec{i}$. We claim that the structure
\[
(n^m, <_{\rm pr}\res n^m, \vec{<}_{{\rm lx}, \vec{i}}\res n^m)
\]
does the job. Color with $d$ colors all substructures of this structure isomorphic to $(X, P^X, \vec{L}^X)$.
By Lemma~\ref{L:two}(i), this induces a coloring of all $\sigma \in {{n}\choose{|X|}}^m\times \binom{m, \vec{i}}{A,\vec{a} }_{\rm rs}$ by coloring
$\sigma$ with the color
of the structure $\pi^\sigma(X)$. By our choice of $m$, $n$, and $\vec{i}$, there exists $\tau_0 \in  {{n}\choose{|Y|}}^m\times
\binom{m, \vec{i}}{B ,\vec{b} }_{\rm rs}$ such that all $\tau_0 \cdot \sigma$,
with $\sigma \in {{Y}\choose{|X|}}^m\times \binom {B,\vec{b}}{A,\vec{a}}_{\rm rs}$,
get the same color. Consider the structure
\[
\pi^{\tau_0}(Y) \subseteq n^m.
\]
By Lemma~\ref{L:two}(i), it is isomorphic to $(Y, P^Y, \vec{L}^Y)$. By Lemma~\ref{L:two}(ii), each substructure of $\pi^{\tau_0}(Y)$ that is isomorphic to $(X, P^X, \vec{L}^X)$ is of
the form $\pi^{\tau_0\cdot \sigma}(X)$ for $\sigma \in {{Y}\choose{|X|}}^m\times \binom {B,\vec{b}}{A,\vec{a} }_{\rm rs}$.
So all of them have the same color.

\section {On the relationship between Propositions \ref{P:VProduct} and \ref{P:TwistProduct} and Theorem \ref{T}}\label{S:redo}

The arguments in Sections~\ref{S:str} and \ref{S:pr} show that Theorem \ref{T} is, in a sense, a translation of Proposition~\ref{P:TwistProduct},
which is, in a sense, a particular case of  Proposition~\ref{P:VProduct}.
In the present section, we make the notion of translation mathematically precise using a variation of the concept of interpretation from \cite{Sol}. 
Even though this material is not necessary for understanding the proof of Theorem~\ref{T} as presented in the previous sections, it seems 
worthwhile to place this proof in a broader context. 

As argued in \cite{Sol},
many particular Ramsey statements are instances of a general Ramsey statement formulated for certain algebraic structures. Interpretation
is a precise notion of ``homomorphism" that allows one to transfer the Ramsey statement from one such algebraic structure to another. We explain 
details of this setup below. Further, we define such algebraic structures for the statements in Proposition~\ref{P:TwistProduct} and Theorem~\ref{T} 
and show that the first one interprets the second one. So Propositions~\ref{P:VProduct} and \ref{P:TwistProduct} are the Ramsey theoretic essence 
of the main result Theorem~\ref{T}.

Consider a set $A$ with a partial function from $A\times A$ to $A$: $(a,b)\to a\cdot b$.
Let $\mathcal F$ and $\mathcal R$ be families of subsets of $A$. Let $(F,R)\to F\bullet R$ be a function whose domain is a subset of
$\mathcal F\times \mathcal R$, whose values are subsets of $A$, and which is such that whenever $F\bullet R$ is defined, then $f\cdot r$ is defined for all $f\in F$
and $r\in R$ and $F\bullet R=\{f\cdot r\colon f\in F, r\in R \}$. We say that $(\mathcal F,\mathcal R,\bullet)$ is a {\em pair of families over} $(A, \cdot)$.

Let $(\mathcal F,\mathcal R,\bullet)$ and $(\mathcal G,\mathcal S,\bullet)$ be pairs of families over $(A,\cdot)$ and $(B,\cdot)$, respectively. We say that
$S\in {\mathcal S}$ is {\em interpretable in} $(\mathcal F,\mathcal R)$ if there exists $R\in \mathcal R$ and a function $\alpha:S\to R$ such that
if $F\bullet R$ is defined for $F\in \mathcal F$, then there exists $G\in \mathcal G$ with $G\bullet S$ defined, and a function $\phi:F\to G$
such that  for  $f_1,f_2\in F$ and $s_1,s_2\in S$,
\begin {equation} \label {E:interpretation}
f_1\cdot \alpha(s_1)=f_2\cdot \alpha(s_2)\Longrightarrow \phi(f_1)\cdot s_1=\phi(f_2)\cdot s_2.
\end {equation}

Now, we formulate the Ramsey condition for a pair of families.
Let $(\mathcal F,\mathcal R,\bullet)$ be a pair of  families and let $d>0$. We say the $d$-{\em Ramsey condition} holds for $(\mathcal F,\mathcal R,\bullet)$ if
for each $R\in \mathcal R$, there exists $F\in \mathcal F$ such that for each $d$-coloring of $F\bullet R$, there exists $f\in F$ with $\{f\cdot r:r\in R\}$ is monochromatic.

The following proposition can be checked without difficulty.
\begin {proposition}\label {T:Weakinterpretation}
Let $(\mathcal F,\mathcal R,\bullet)$ and $(\mathcal G,\mathcal S,\bullet)$ be pairs of families, and let $d>0$.
If the $d$-Ramsey condition holds for $(\mathcal F,\mathcal R,\bullet)$ and each $S\in \mathcal S$ is interpretable in $(\mathcal F,\mathcal R,\bullet)$,
then the $d$-Ramsey condition holds for $(\mathcal G,\mathcal S,\bullet)$.
\end {proposition}

From now on, we fix $d$, the number of colors.

\textbf{A pair of families for Proposition \ref {P:TwistProduct}.}
Let $A_1$ consist of all $\tau$ belonging to ${\N \choose{l}}^m\times \binom{C,\vec{c}} {B,\vec b}_{\rm rs}$ for some natural numbers $l$ and  $m$, $B\subseteq {\rm lin}_l$,
a linearly ordered set $C$, and
anchored sequences $\vec b$ and $\vec c$ of length $p$ of elements of $B$ and $C$, respectively.
If $\sigma, \tau\in A_1$, then $\tau\cdot \sigma$ is defined precisely when
$\sigma \in {l \choose {k}}^{m}\times \binom{B, \vec{b}}{A, \vec{a}}_{\rm rs}$ and $\tau\in {\N \choose{l}}^{m}\times \binom{m,\vec{i}} {B,\vec b}_{\rm rs}$, and $\tau\cdot \sigma$ is
defined by formula \eqref{E:TwistComp}.

Let $\mathcal F_1$ consist of all sets of the form $F= {n \choose{l}}^q\times \binom{q,\vec i} {C,\vec c}_{\rm rs}$ for some natural numbers $l\leq n$ and  $q$, $C\subseteq {\rm lin}_l$,
and anchored sequences $\vec i$ and $\vec c$ of length $p$ of elements of $q$ and $C$, respectively.
Let ${\mathcal S}_1^d$ consist of all sets of the form $S= \binom{r}{k}^m\times \binom{B, {\vec b}}{A, {\vec a}}_{\rm rs}$ for some natural numbers $k\leq r$, $m$, $A\subseteq {\rm lin}_k$ and
$B\subseteq {\rm lin}_r$ and anchored sequences $\vec a$, $\vec b$ of length $p$ of elements of $A$ and $B$, respectively, where $m$ is large enough so that the Dual Ramsey Theorem
with Constants, as stated in Section~\ref{S:prodth}, holds with $d$ colors for $m$, $(A, \vec{a})$ and $(B, \vec{b})$.
For $F\in \mathcal F_1$ and $S\in {\mathcal S}_1^d$ as above, $F\bullet S$ is defined when $r=l$, $q=m$, and $(B, {\vec b}) = (C, {\vec c})$, and is then equal to
$\{\tau \cdot \sigma\colon \tau\in F, \sigma \in S \}$.

Following the proof of Proposition~\ref{P:TwistProduct} one gets Proposition~\ref{T:VProduct} below.
\begin {proposition}\label {T:VProduct}
The $d$-Ramsey condition holds for $(\mathcal F_1,{\mathcal S}_1^d,\bullet)$.
\end {proposition}

\textbf{A pair of families for Theorem \ref{T}.}
Let $A_2$ consist of all embeddings between structures of the form $(X, P^X, \vec{L}^X)$ as in Section~\ref{S:theo}. 
For $f,g\in A_2$, $f\cdot g$ is defined precisely when the domain structure of $f$ is equal to the range structure of $g$ and 
then we let $f\cdot g = f\circ g$.

Let $\mathcal F_2$ consist of all sets $F={{n^m,<_{\rm pr},<_{{\rm lx},\vec i}}\choose{Y,P^Y,\vec L^Y}}$, and let $\mathcal S_2$
consist of all
$S= {{Z,P^Z,\vec L^Z}\choose {X,P^X,\vec L^X}}$. For $F\in \mathcal F_2$ and $S\in \mathcal S_2$ as above, $F\bullet S$ is defined precisely when
$(Y,P^Y,\vec L^Y)=(Z,P^Z,\vec L^Z)$ and is then equal to $\{f\cdot g\colon f\in F, g\in S \}$.

\begin {proposition}\label {P:Weakinterpretation}
For each $d>0$, each $S\in \mathcal S_2$ is interpretable in $(\mathcal F_1,{\mathcal S}^d_1,\bullet)$.
\end {proposition}
\begin {proof}
Let $S\in \mathcal S_2$ be ${{Y,P^Y,\vec L^Y}\choose {X,P^X,\vec L^X}}$.
Set $k=|X|$ and $l=|Y|$.
Observe that we can assume $(X,P^X,\vec L^X)=(k,P^k,\vec L^k)$ and $(Y,P^Y,\vec L^Y)=(l,P^l,\vec L^l)$ where $P^k,P^l$ are partial orders on $k,l$ and $\vec{L}^k,\vec{L}^l$
are sequences of linear orders of length $p$ extending $P^k,P^l$ with $L_0^k=<\upharpoonright k$, $L_0^l=<\upharpoonright l$,  respectively. Fix $m$ such that
$R={l \choose k}^m\times \binom {{\rm lin}_{L_0^l}(P^l),\vec L^l}{{\rm lin}_{L_0^k}(P^k),\vec L^k}_{\rm rs}$ is in ${\mathcal S}^d_1$. Define $\alpha: S\to R$ by letting, for $s\in S$,
\begin {equation}\label {E:alpha}
\alpha(s)=(s[k],\dots,s[k],r\circ {\rm res}_{s[k]})
\end {equation}where $r:{\rm lin}_{L_0^l\res s[k]}(P^l\res s[k])\to {\rm lin}_{L_0^k}(P^k)$ is the unique isomorphism. By Lemma~\ref{L:rig}, $\alpha(s)\in R$.

Assume $F\bullet R$ is defined. Then $F= {n \choose{l}}^m\times \binom {m,\vec i} {{\rm lin}_{L_0^l}(P^l),\vec L^l}_{\rm rs}$ for some $n$
and an anchored sequence $\vec i$ of length $p$ of elements of $m$.
Let $G= \binom {n^m,<_{\rm pr},<_{{\rm lx},\vec i }}{l,P^l,\vec L^l}$. So $G\bullet S$ is defined. Define
$\phi\colon F\to G$ by $\phi(\tau)=\pi^\tau$, where $\pi^\tau$ is as in \eqref{E:pitau}. Note that by Lemma \ref{L:two}(i), $\pi^\tau\in G$.

If $\tau=(T_0,\dots,T_{m-1},t)\in F$ and $s\in S$, then by (\ref {E:TwistComp}) and (\ref {E:alpha}) we have
\begin {equation}\notag
\tau\cdot \alpha(s)=((\pi^{\tau}_0\circ s)[k],\dots,(\pi^{\tau}_{m-1}\circ s)[k],r\circ {\rm res}_{s[k]}\circ t).
\end {equation}
Now, one checks, using \eqref{E:pitau}, that $\pi^{\tau\cdot \alpha(s)}=\phi(\tau)\cdot s$, which implies \eqref{E:interpretation}, as required.
\end {proof}

By Propositions~\ref {T:Weakinterpretation}, \ref{T:VProduct} and \ref{P:Weakinterpretation}, the $d$-Ramsey condition holds for $(\mathcal F_2,\mathcal S_2,\bullet)$, so
Theorem~\ref{T} follows.

\end{document}